\documentclass[11pt]{article}
\usepackage{tikz}
\usepackage{nameref}
\usepackage[shortlabels]{enumitem}
\usepackage{empheq}
\usepackage[shortlabels]{enumitem}
\setlist[enumerate]{nosep}
\usepackage[doc]{optional}
\usepackage{xcolor}
\usepackage[colorlinks=true,
            linkcolor=refkey,
            urlcolor=lblue,
            citecolor=red]{hyperref}
\usepackage{float}
\usepackage{soul}
\usepackage{graphicx}
\definecolor{labelkey}{rgb}{0,0.08,0.45}
\definecolor{refkey}{rgb}{0,0.6,0.0}
\definecolor{Brown}{rgb}{0.45,0.0,0.05}
\definecolor{lime}{rgb}{0.00,0.8,0.0}
\definecolor{lblue}{rgb}{0.5,0.5,0.99}

 \usepackage{mathpazo}

\usepackage{stmaryrd}
\usepackage{amssymb}
\makeatletter
\def\namedlabel#1#2{\begingroup
   \def\@currentlabel{#2}%
   \label{#1}\endgroup
}
\makeatother

\newcommand{\seppthree}{\setlength{\itemsep}{-3pt}}

\newcommand{\reckti}{$3^*$ monotone}

\newcommand{\J}[1]{\ensuremath{{\operatorname{J}}_%
{#1}}}
\newcommand{\R}[1]{\ensuremath{{\operatorname{R}}_%
{#1}}}
\newcommand{\Pj}[1]{\ensuremath{{\operatorname{P}}_%
{#1}}}

\providecommand{\siff}{\Leftrightarrow}

\newcommand{\nnn}{\ensuremath{{n\in{\mathbb N}}}}

\newcommand{\menge}[2]{\big\{{#1}~\big |~{#2}\big\}}

\newcommand{\To}{\ensuremath{\rightrightarrows}}

\newcommand{\fenv}[1]%
{\ensuremath{\,\overrightarrow{\operatorname{env}}_{#1}}}
\newcommand{\benv}[1]%
{\ensuremath{\,\overleftarrow{\operatorname{env}}_{#1}}}

\newcommand{\scal}[2]{\left\langle{#1},{#2}  \right\rangle}

\newcommand{\RR}{\ensuremath{\mathbb R}}

\newcommand{\dom}{\ensuremath{\operatorname{dom}}}

\newcommand{\inte}{\ensuremath{\operatorname{int}}}

\newcommand{\ran}{\ensuremath{\operatorname{ran}}}

\newcommand{\Id}{\ensuremath{\operatorname{Id}}}

\newcommand{\mdv}[1]{\ensuremath{{\operatorname{v}}_%
{#1}}}

\providecommand{\wtA}{\widetilde{A}}
\providecommand{\wtB}{\widetilde{B}}

{\begin{list}{}{%
\settowidth{\labelwidth}{\textrm{#1~}}%
\setlength{\leftmargin}{\labelwidth+\labelsep}}}
{\end{list}}
\usepackage{amsthm}
\usepackage[capitalize,nameinlink]{cleveref}
\crefname{equation}{}{equations}
\crefname{chapter}{Appendix}{chapters}
\crefname{item}{}{items}
\crefname{enumi}{}{}
\newtheorem{theorem}{Theorem}[section]
\newtheorem{lemma}[theorem]{Lemma}

\newtheorem{corollary}[theorem]{Corollary}

\newtheorem{proposition}[theorem]{Proposition}

\newtheorem{example}[theorem]{Example}
\newtheorem{ex}[theorem]{Example}
\newtheorem{fact}[theorem]{Fact}
\newtheorem{remark}[theorem]{Remark}

\providecommand{\RA}{\Rightarrow}

\providecommand{\RR}{\mathbb{R}}

\providecommand{\ran}{\operatorname{ran}}

\providecommand{\dom}{\operatorname{dom}}

\providecommand{\gr}{\operatorname{gra}}
\providecommand{\gra}{\operatorname{gra}}
\providecommand{\Id}{\operatorname{{ Id}}}

\providecommand{\To}{\rightrightarrows}

\providecommand{\gr}{\operatorname{gra}}

\providecommand{\ran}{\operatorname{ran}}

\providecommand{\Id}{\operatorname{Id}}

\providecommand{\J}{{ J}}
\providecommand{\R}{{ R}}

\providecommand{\cran}{\overline{\ran}}

\providecommand{\RR}{\mathbb{R}}

\definecolor{myblue}{rgb}{.8, .8, 1}
  \newcommand*\mybluebox[1]{%
    \colorbox{myblue}{\hspace{1em}#1\hspace{1em}}}

\allowdisplaybreaks 

\begin{document}

%

\author{
Heinz H.\ Bauschke\thanks{
Mathematics, University
of British Columbia,
Kelowna, B.C.\ V1V~1V7, Canada. E-mail:
\texttt{heinz.bauschke@ubc.ca}.}
~and~ Walaa M.\ Moursi\thanks{
  Department of Electrical Engineering,
  Stanford University,
  350 Serra Mall, Stanford, CA 94305,
  USA.
  E-mail: \texttt{wmoursi@stanford.edu}.}
}

\title{\textsf{
On the minimal displacement vector of 
compositions\\ and convex combinations of nonexpansive mappings
}
}


\date{September 4, 2018}

\maketitle

\begin{abstract}
\noindent
Monotone operators and (firmly) nonexpansive mappings are fundamental
objects in modern analysis and computational optimization. 
Five years ago, it was shown that if finitely many firmly
nonexpansive mappings
have or ``almost have'' fixed points, then the same is true for
compositions and convex combinations.
More recently, sharp information about the minimal displacement
vector of compositions and of convex combinations of firmly nonexpansive
mappings was obtained in terms of
the displacement vectors of the underlying operators. 

Using a new proof technique based on the Brezis--Haraux theorem and 
reflected resolvents, we extend these results
from firmly nonexpansive to general averaged nonexpansive mappings. 
Various examples illustrate the tightness of our results. 
\end{abstract}
{ 
\noindent
{\bfseries 2010 Mathematics Subject Classification:}
{Primary 
47H05, 
47H09; 
Secondary 
47H10, 
90C25. 
}

\noindent {\bfseries Keywords:}
Averaged nonexpansive mapping, 
Brezis--Haraux theorem, 
displacement map, 
maximally monotone operator,
minimal displacement vector, 
nonexpansive mapping,
resolvent. 
}

\section{Introduction}

Throughout, we assume that 
\begin{empheq}[box=\mybluebox]{equation}
\text{$X$ is
a real Hilbert space with inner
product $\scal{\cdot}{\cdot}$ }
\end{empheq}
and induced norm $\|\cdot\|$.
Recall that 
$T\colon X\to X$ is 
\emph{nonexpansive} (i.e., \emph{$1$-Lipschitz continuous)}
if $(\forall (x,y)\in X\times X)$ 
$\|Tx-Ty\|\leq \|x-y\|$ and that it 
is \emph{firmly nonexpansive} 
if $(\forall (x,y)\in X\times X)$ 
$\|Tx-Ty\|^2\leq \scal{x-y}{Tx-Ty}$.
Furthermore, recall that 
a set-valued operator $A\colon X\To X$ is 
\emph{maximally monotone} if it is \emph{monotone}, i.e., 
$\{(x,x^*),(y,y^*)\}\subseteq \gra A
\RA
\scal{x-y}{x^*-y^*}\geq 0$ and if the graph of $A$ cannot be properly
enlarged without destroying monotonicity\footnote{
We shall write $\dom A = \menge{x\in X}{Ax\neq\varnothing}$
for the \emph{domain} of $A$, $\ran A = A(X) = \bigcup_{x\in X}Ax$ for
the \emph{range} of $A$, and 
$\gr A=\menge{(x,u)\in X\times X}{u\in Ax}$ 
for the \emph{graph} of
$A$.}. 
These notions are of central importance in modern optimization;
see, e.g., \cite{BC2017}, \cite{GK}, 
\cite{GR}, and the references therein. 
Maximally monotone operators and firmly nonexpansive mappings
are closely related to each other (see \cite{Minty} and \cite{EckBer}) 
because 
if $A\colon X\To X$ is maximally monotone, then
its \emph{resolvent} 
\begin{empheq}[box=\mybluebox]{equation}
\J{A} = (\Id+A)^{-1}
\end{empheq} 
is firmly nonexpansive, and if $T\colon X\to X$ is firmly nonexpansive,
then $T^{-1}-\Id$ is maximally monotone\footnote{
Here and elsewhere, $\Id$ denotes the \emph{identity} operator on $X$.}.
In a similar vein,
the classes of firmly nonexpansive and simply nonexpansive mappings
are bijectively linked because the \emph{reflected resolvent}
\begin{empheq}[box=\mybluebox]{equation}
\R{A} = 2\J{A}-\Id
\end{empheq} 
is nonexpansive, and every nonexpansive map arises in this way. 

The interest into these operators stems from the fact that
minimizers of convex functions are zeros of maximally monotone operators
which in turn are fixed points of (firmly) nonexpansive mappings. 
For basic background material in fixed point theory and monotone operator
theory, we refer the reader to 
\cite{BC2017},
\cite{Brezis},
\cite{BurIus},
\cite{GK},
\cite{GR}, 
\cite{Rock70},
\cite{Rock98},
\cite{Simons1},
\cite{Simons2},
\cite{Zalinescu},
\cite{Zeidler1},
\cite{Zeidler2a},
and \cite{Zeidler2b}.

However, not every problem has a solution; equivalently, not
every resolvent has a fixed point.
Let us make this concrete by 
assuming that $R\colon X\to X$ is  nonexpansive.
The deviation of $T$ possessing a fixed point is captured by the
notion of the \emph{minimal displacement vector}
which is well defined by\footnote{ Given a nonempty closed convex
subset $C$ of $X$, we denote its \emph{projection mapping} or projector
by $\Pj{C}$.}
\begin{empheq}[box=\mybluebox]{equation}
\mdv{R} =
P_{\overline{\ran}(\Id-R)}(0). 
\end{empheq}
If $\mdv{R}=0$, then either $R$ has a fixed point or
$R$ ``almost'' has a fixed point in the sense that there
is a sequence $(x_n)_\nnn$ in $X$ such that $x_n-Rx_n\to 0$. 

Let us now assume that 
$m\in\{2,3,4,\ldots\}$
and that 
we are given $m$ nonexpansive operators $R_1,\ldots,R_m$ on $X$, 
with corresponding 
minimal displacement vectors 
$\mdv{R_i}$.
A natural question is the following: 
\begin{quotation}
\noindent
{\em 
What can be said about the minimal displacement vector $\mdv{R}$ of
$R$, when $R$ is either a composition or a convex combination of
$R_1,\ldots,R_m$, in terms of the given minimal displacement vectors 
$\mdv{R_1},\ldots,\mdv{R_m}$?
}
\end{quotation}
Five years ago, the authors of \cite{Vic1} proved the
that if 
each $\mdv{R_i}=0$, then so is $\mdv{R}=0$ (in the composition
and the convex combination case) \emph{provided that 
each $R_i$ is firmly nonexpansive} (see also \cite{Bau03} for the
earlier case when each $R_i$ is a projector). 
It is noteworthy that these results have been studied 
fairly recently by Kohlenbach in \cite{KGA17} and \cite{Kohlen17} 
from the viewpoint of \emph{proof mining}. 
In the past year, these results were extended in \cite{Vic2} to derive
bounds on the displacement vector, \emph{but still under the assumption of
firm nonexpansiveness}. 
\begin{quotation}
\noindent
{\em 
In this paper we obtain precise information on
the minimal displacement vector $\mdv{R}$ under the much
less restrictive assumption that each $R_i$ 
is merely averaged (rather than firmly)
nonexpansive. 
}
\end{quotation}

The important class of averaged nonexpansive mappings 
(see the comprehensive
study \cite{Comb04} for more) is much larger
than the class of firmly nonexpansive mappings. Indeed, the former class
is closed under compositions (but not the latter) and every nonexpansive
mapping can be approximated by a sequence of averaged nonexpansive
mappings. The key tool to derive our results is the celebrated
Brezis--Haraux theorem \cite{Br-H}, which is applied in a completely
novel way in this work.

 Our new results, outlined next,  
  massively generalize the results in 
  \cite{Vic1} and \cite{Vic2}
  in various directions:

  \begin{itemize}
    \item[{\bf R1}]
    \namedlabel{D:1}{\bf R1}
    We obtain very
    powerful formulae 
    for the ranges of the 
    displacement mapping of compositions 
    and convex combinations. 
    These formulae precisely describe the 
    closure of the range 
    displacement mapping of compositions 
    and convex combinations
    in terms of the closures of the 
    ranges of the displacement mappings of
    the individual operators (see 
    \eqref{e:aug14iandii} and \eqref{e:aug16i}).

    \item[{\bf R2}]
    \namedlabel{D:2}{\bf R2}
    Regarding the minimal displacement vector 
    of compositions, we relax the assumption that 
    all mappings are firmly nonexpansive to
    all but one map are averaged nonexpansive
    (see Theorem~\ref{s:compo}).

    \item[{\bf R3}]
    \namedlabel{D:3}{\bf R3}
    We show that the conclusion of 
    \ref{D:1} is sharp, by providing a counterexample 
    when more than 
    one map fail to be averaged
    (see Example~\ref{ex:all:but:1:fail}).

    \item[{\bf R4}]
    \namedlabel{D:4}{\bf R4}
    Regarding the minimal displacement vector
    of convex combinations, we relax the assumption that 
    all mappings are firmly nonexpansive to
    them being merely nonexpansive
    (see Theorem~\ref{thm:main:convc}).

    \item[{\bf R5}]
    \namedlabel{D:5}{\bf R5}
    We discuss the attainment of the gap vector 
    of the compositions and the connection to cyclic 
    and noncyclic shifts of the compositions
    (see Proposition~\ref{p:aug15iii}
    and Remark~\ref{rem:De:Pierro}). 
  \end{itemize}

The remainder of this paper is organized as follows.
In Section~\ref{s:aux}, we collect various auxiliary results
which will make the proofs of the main results more structured and pleasant.
We then turn to compositions of two mappings in Section~\ref{s:twocompo}.
The new main results
concerning compositions are presented in
Section~\ref{s:compo} while convex combinations are dealt with in
Section~\ref{s:convco}. 

Finally, 
our notation is standard and follows \cite{BC2017} to which we
also refer for facts not explicitly mentioned here. 

\section{Auxiliary results}
\label{s:aux}

This section contains various results that will aid in the derivation of the main
results in subsequent sections.

\subsection{Resolvents and reflected resolvents}

Let $C\colon X\To X$ be maximally monotone. 
The \emph{inverse resolvent identity} (see, e.g., \cite[(23.17)]{BC2017} or 
\cite[Lemma~12.14]{Rock98})
\begin{equation}
  \label{e:invresid}
\J{C^{-1}} = \Id-\J{C} 
\end{equation}
is fundamental, as is the \emph{Minty parametrization}
(see, e.g., \cite[Remark~23.23(ii)]{BC2017})
\begin{equation}
\label{e:Mintypar}
\gra C = \menge{(\J{C}x,\Id-\J{C}x)}{x\in X}
= \menge{(\J{C}x,\J{C^{-1}}x)}{x\in X}.
\end{equation}
Because the reflected resolvent is 
$  \R{C} = 2\J{C}-\Id$, 
we obtain $\Id-\R{C} = \Id-(2\J{C}-\Id)=2(\Id-\J{C})$
and further not only 
\begin{equation}
\label{e:voulezvous}
2\ran(C)=2\ran(\Id-\J{C})=\ran(\Id-\R{C})
\end{equation}
by using \eqref{e:Mintypar}, but also 
\begin{equation}
\R{C} = \J{C} - \J{C^{-1}}.
\end{equation}
Let $x$ and $y$ be in $X$. 
Then
$\J{-y+C}(x)= \J{C}(x+y)$
and 
$\J{C(\cdot-y)}(x) = y+\J{C}(x-y)$.
Hence
$\R{-y+C}(x)= 2\J{C}(x+y)-x=y+\R{C}(x+y)$
and 
$\R{C(\cdot-y)}(x) = y+\R{C}(x-y)$.
It follows that
$\R{B(\cdot-y)}\R{-y+A}x = y+\R{B}\R{A}(x+y)$
and that 
$ -2y+(x+y)-\R{B}\R{A}(x+y) =x-\R{B(\cdot-y)}\R{-y+A}(x)$.
This yields the useful translation formula
\begin{equation}
  \label{e:aug17ii}
  \ran(\Id-\R{B}\R{A})=2y+\ran(\Id-\R{B(\cdot-y)}\R{-y+A}). 
\end{equation}

\subsection{Averaged nonexpansive mappings}
Let $R  \colon X\to X$ and let $\alpha\in \left[
0,1\right[$.
Recall that
$R$ is $\alpha$-averaged if
$R=(1-\alpha)\Id+\alpha N$,
where $N$ is nonexpansive; 
equivalently, by \cite[Lemma~2.1]{Comb04} 
or \cite[Proposition~4.35]{BC2017}:
\begin{equation}
(\forall x\in X)(\forall y\in X)\quad
\|(\Id-R)x-(\Id-R)y\|^2 
\leq \tfrac{\alpha}{1-\alpha}\big(\|x-y\|^2 - \|Rx-Ry\|^2\big).
\end{equation}
If we don't wish to stress the constant $\alpha$ we refer to $R$ simply
as averaged or averaged nonexpansive. 
We have the following useful result.
\begin{fact}
\label{fact:ave:comp}
Let $m\in\{2,3,\ldots\}$, 
and let $R_1,\ldots,R_m$ be averaged on $X$.
Then $R_m\cdots R_1$ is also averaged.
\end{fact}
\begin{proof}
See \cite[Proposition~2.5]{CombYam14}.
\end{proof}

\subsection{Cocoercive operators}

Let $\mu>0$ and let 
$A\colon X\to X$.
Then $A$ is \emph{$\mu$-cocoercive} if
$\mu A$ is firmly nonexpansive, i.e.,
\begin{equation}
(\forall x\in X)(\forall y\in X) \quad
\scal{x-y}{Ax-Ay}\geq \mu\|Ax-Ay\|^2;
\end{equation}
equivalently, 
$A^{-1}$ is $\mu$-strongly monotone, i.e., 
$A^{-1}-\mu\Id$ is monotone
(see \cite[Example~22.7]{BC2017}).

The following result is implicitly contained 
in Moursi and Vandenberghe's 
\cite[Proposition~2.1(iii)]{MV18},
and it extends previous work by Giselsson 
\cite[Proposition~5.3]{Gis17}. 

\begin{proposition}
\label{p:AcocoRaver}
Let
$A\colon X\To X$ be maximally monotone, and let $\mu>0$. 
Then $A$ is $\mu$-cocoercive
if and only if $\R{A}$ is $(1+\mu)^{-1}$-averaged. 
\end{proposition}
\begin{proof}
It is straightforward to verify that 
\begin{equation}
\label{e:katibday3}
(\forall u\in X)(\forall v\in X) \quad
4\big(\scal{v}{u-v}-\mu\|u-v\|^2\big) 
= \|u\|^2 - \|2v-u\|^2 - 4\mu\|u-v\|^2
\end{equation}
Using the Minty parametrization (see \eqref{e:Mintypar}), we see that
\begin{multline}
\label{e:katibday1}
\text{$A$ is $\mu$-cocoercive}
\siff\\
(\forall x\in X)(\forall y\in X)\;\;
\scal{\J{A}x-\J{A}y}{(x-y)-(\J{A}x-\J{A}y)}
\geq
\mu\|(x-y)-(\J{A}x-\J{A}y)\|^2\siff\\
(\forall x\in X)(\forall y\in X)\;\;
4\big(\scal{\J{A}x-\J{A}y}{(x-y)-(\J{A}x-\J{A}y)}
- \mu\|(x-y)-(\J{A}x-\J{A}y)\|^2\big) \geq 0.
\end{multline}
On the other hand, 
using \cite[Proposition~4.35]{BC2017}, we have
\begin{multline}
\label{e:katibday2}
\text{$\R{A}$ is $(1+\mu)^{-1}$-cocoercive}
\siff\\
(\forall x\in X)(\forall y\in X)\;
\|\R{A}x-\R{A}y\|^2 \leq \|x-y\|^2 - \frac{1-(1+\mu)^{-1}}{(1+\mu)^{-1}}
\|(x-y)-(\R{A}x-\R{A}y)\|^2\siff\\
(\forall x\in X)(\forall y\in X)\;
\|x-y\|^2 - \|2(\J{A}x-\J{A}y)-(x-y)\|^2 -  4\mu \|(x-y)-(\J{A}x-\J{A}y)\|^2\geq 0.
\end{multline}
Now combine \eqref{e:katibday1}, \eqref{e:katibday2},
and \eqref{e:katibday3} with $u=x-y$ and $v=\J{A}x-\J{A}y$. 
\end{proof}

\begin{lemma}
\label{l:B-H}
Let $A$ and $B$ be maximally monotone on $X$.
Suppose that there exists $C$ in $\{A,B\}$ such that
$C\colon X\to X$ is cocoercive. 
Then 
$\cran(A+B) = \overline{\ran A + \ran B}$
and $\inte\ran(A+B) = \inte(\ran A+\ran B)$. 
\end{lemma}
\begin{proof}
Because $\dom C = X$, the sum rule 
(see, e.g., \cite[Corollary~25.5(i)]{BC2017}) 
yields the maximal monotonicity of $A+B$. 
Moreover, $C$ is \reckti\ by \cite[Example~25.20(i)]{BC2017}.
Altogether, the conclusion follows from
the Brezis--Haraux theorem (see, e.g., \cite[Theorem~25.24(ii)]{BC2017}). 
\end{proof}

\subsection{On the range of a displacement map}

Let  $A$ and $B$ be maximally monotone operators on $X$.
Because of 
\begin{equation}
\label{e:siskommt1}
\Id-\R{B}\R{A} = 2\J{A}-2\J{B}\R{A}
= 2\J{A}-2\J{B}(\J{A}-\J{A^{-1}})
= 2\Id-2\J{A}+2(\Id-\J{B})\R{A},
\end{equation} 
we have 
$\ran(\Id-\R{B}\R{A}) \subseteq 2\ran(\Id-\J{A}+(\Id-\J{B})\R{A})
\subseteq \ran(\Id-\R{A}) + \ran(\Id-\R{B})$ by
\eqref{e:voulezvous}.
It follows that
\begin{equation}
\label{e:supertruper}
\cran(\Id-\R{B}\R{A}) \subseteq 
\overline{\ran(\Id-\R{A}) + \ran(\Id-\R{B})}. 
\end{equation}

\section{Composition of two mappings}
\label{s:twocompo}

In this section, we study the composition of two mappings.

\begin{lemma}
\label{l:Key}
Let $A$ and $B$ be maximally monotone on $X$.
Suppose that there exists $C$ in $\{A,B\}$
such that 
$C\colon X\to X$ and $C$ is cocoercive. 
Then the following hold:
\begin{enumerate}
\item 
\label{l:Keyi}
$\overline{\ran(\Id-\R{A})+ \ran(\Id-\R{B})} \subseteq \cran(\Id-\R{B}\R{A})$.
\item 
\label{l:Keyii}
$\overline{\ran(\Id-\R{A})+ \ran(\Id-\R{B})} \subseteq \cran(\Id-\R{A}\R{B})$.
\end{enumerate}
\end{lemma}
\begin{proof}
We start with by establishing the following

\texttt{Claim:}
If $\wtA$ and $\wtB$ are maximally
monotone on $X$, and $0\in\cran(\wtA+\wtB)$, 
then $0\in\cran(\Id-\R{\wtB}\R{\wtA})$.

To this end, 
assume there exist sequences
$(x_n,u_n)_\nnn$ in $\gr \wtA$
and 
$(x_n,v_n)_\nnn$ in $\gr \wtB$
such that
\begin{equation}
  \label{e:siskommt4}
u_n+v_n\to 0.
\end{equation}
The Minty parametrizations (see \eqref{e:Mintypar}) of $\gra \wtA$ and $\gr \wtB$ give 
\begin{equation}
  \label{e:siskommt3}
(\forall\nnn)\;\;
x_n=\J{\wtA}(x_n+u_n), \; u_n=\J{\wtA^{-1}}(x_n+u_n)\;\text{and}\;
v_n=\J{\wtB^{-1}}(x_n+v_n);
\end{equation}
hence, $x_n-u_n=(\J{\wtA}-\J{\wtA^{-1}})(x_n+u_n)=\R{\wtA}(x_n+u_n)$. 
Set
\begin{equation}
  \label{e:siskommt5}
(\forall\nnn)\;\;
z_n = 2\J{\wtA}(x_n+u_n)-2\J{\wtB}\R{\wtA}(x_n+u_n)
\in2\ran(\J{\wtA}-\J{\wtB}\R{\wtA}) \stackrel{\eqref{e:siskommt1}}{=} \ran(\Id-\R{\wtB}\R{\wtA}).
\end{equation}
Thus
\begin{equation}
  \label{e:siskommt2}
(\forall\nnn)\;\;
z_n=2x_n-2\J{\wtB}(x_n-u_n).
\end{equation}
Next, on the one hand,
\begin{align}
  (\forall\nnn)\;\;
z_n-2(u_n+v_n)
&\stackrel{\eqref{e:siskommt2}}{=} 2(x_n-u_n) - 2\J{\wtB}(x_n-u_n)-2v_n
\stackrel{\eqref{e:invresid}}{=} 2\J{\wtB^{-1}}(x_n-u_n)-2v_n\notag\\
&\stackrel{\eqref{e:siskommt3}}{=} 2\J{\wtB^{-1}}(x_n-u_n)-2\J{\wtB^{-1}}(x_n+v_n).
\end{align}
On the other hand,
as $\J{\wtB^{-1}}$ is nonexpansive, we have
$\|\J{\wtB^{-1}}(x_n-u_n)-\J{\wtB^{-1}}(x_n+v_n)\|
\leq \|(x_n-u_n)-(x_n+v_n)\|=\|u_n+v_n\|\to 0$ by \eqref{e:siskommt4}.
Altogether,
$z_n-2(u_n+v_n)\to 0$ and thus $z_n\to 0$ by \eqref{e:siskommt4}.
Recalling that $(z_n)_\nnn$ lies in $\ran(\Id-\R{\wtB}\R{\wtA})$ 
from \eqref{e:siskommt5},
we finally deduce the \texttt{Claim} that $0\in\cran(\Id-\R{\wtB}\R{\wtA})$. 

Having established the \texttt{Claim}, we now let $y\in X$.

\emph{Case~1:} $C=B$. \\
\ref{l:Keyi}:
Indeed, the following implications
\begin{align*}
y\in \overline{\ran(\Id-\R{A}) + \ran(\Id-\R{B})}
&\siff
y/2\in \overline{\ran A + \ran B}\tag{by \eqref{e:voulezvous}}\\
&\siff
y/2\in \overline{\ran A + \ran B(\cdot - y/2)}
\tag{as $\ran B = \ran B(\cdot -y/2)$}\\
&\siff 
y/2\in \cran(A+B(\cdot -y/2))
\tag{by Lemma~\ref{l:B-H}}\\
&\siff 
0\in \cran((-y/2+A)+B(\cdot -y/2))\\
&\RA
0\in\cran(\Id-\R{B(\cdot-y/2)}\R{-y/2+A}) \tag{by the \texttt{Claim}}\\
&\siff
0\in-y+\cran(\Id-\R{B}\R{A}) \tag{by \eqref{e:aug17ii}}\\
&\siff
y\in\cran(\Id-\R{B}\R{A})
\end{align*}
yield the conclusion. 

\ref{l:Keyii}:
Similarly to \ref{l:Keyi}, the following implications
\begin{align*}
y\in \overline{\ran(\Id-\R{A}) + \ran(\Id-\R{B})}
&\siff
y/2\in \overline{\ran A + \ran B}\tag{by \eqref{e:voulezvous}}\\
&\siff
y/2\in\overline{\ran A(\cdot-y/2) + \ran B}
\tag{as $\ran A = \ran A(\cdot -y/2)$}\\
&\siff 
y/2\in \cran(A(\cdot -y/2)+B)
\tag{by Lemma~\ref{l:B-H}}\\
&\siff 
0\in \cran(A(\cdot-y/2)+(-y/2+B))\\
&\RA
0\in \cran(\Id-\R{A(\cdot-y/2)}\R{-y/2+B})
\tag{by the \texttt{Claim}}\\
&\siff
0\in-y+\cran(\Id-\R{A}\R{B}) \tag{by \eqref{e:aug17ii}}\\
&\siff
y\in\cran(\Id-\R{A}\R{B})
\end{align*}
yield the conclusion. 

\emph{Case~2:} $C=A$. \\
\ref{l:Keyi}:
Apply item~\ref{l:Keyii} of \emph{Case~1}, with $A,B$ replaced by $B,A$. 
\ref{l:Keyii}:
Apply item~\ref{l:Keyi} of \emph{Case~1}, with $A,B$ replaced by $B,A$. 
\end{proof}

\begin{theorem} 
\label{t:cKey}
Let $A$ and $B$ be maximally monotone.
Suppose that there exists $C\in\{A,B\}$ such that 
$C\colon X\to X$ and $C$ is cocoercive. 
Then 
\begin{equation}
  \overline{\ran(\Id-\R{A})+\ran(\Id-\R{B})}
  =\cran(\Id-\R{B}\R{A})=\cran(\Id-\R{A}\R{B}).
\end{equation}
\end{theorem}
\begin{proof}
Indeed, we have
\begin{align*}
\overline{\ran(\Id-\R{A}) + \ran(\Id-\R{B})}
&\subseteq \cran(\Id-\R{B}\R{A})\cap\cran(\Id-\R{A}\R{B}) \tag{by Lemma~\ref{l:Key}}\\
&\subseteq \cran(\Id-\R{B}\R{A})\cup\cran(\Id-\R{A}\R{B}) \\
&\subseteq \overline{\ran(\Id-\R{A}) + \ran(\Id-\R{B})} \tag{by \eqref{e:supertruper}}.
\end{align*}
Hence all inclusions are in fact equalities and we are done. 
\end{proof}

\begin{theorem}[\rm \textbf{composition of two 
  nonexpansive mappings}]
\label{t:key}
Let $R_1$ and $R_2$ be nonexpansive on $X$,
and suppose that $R_1$ or $R_2$ is actually averaged nonexpansive.
Then
\begin{equation}
\cran(\Id-R_2R_1) = \cran(\Id-R_1R_2) = 
\overline{\ran(\Id-R_1)+\ran(\Id-R_2)}.
\end{equation}
\end{theorem}
\begin{proof}
Because $R_1$ and $R_2$ are nonexpansive,
there exist maximally monotone operators
$A$ and $B$ on $X$ such that
$R_1=\R{A}$ and $R_2=\R{B}$ by using \cite[Corollary~23.9~and~Proposition~4.4]{BC2017}
applied to $\tfrac{1}{2}(\Id+R_1)$ and $\tfrac{1}{2}(\Id+R_2)$. 

\emph{Case~1:} $R_1$ is averaged nonexpansive. \\
By Proposition~\ref{p:AcocoRaver}, $A\colon X\to X$ and $A$ is cocoercive.
The conclusion is now clear from 
Theorem~\ref{t:cKey}.

\emph{Case~2:} $R_2$ is averaged nonexpansive.\\
Argue similarly to \emph{Case~1}, or apply \emph{Case~1} 
with $R_1$ and $R_2$ interchanged. 
\end{proof}

We conclude this section with an $m$-operator version of Theorem~\ref{t:key} (which we will
sharpen in Section~\ref{s:compo}).

\begin{proposition}
\label{p:aug12i}
Let $m\in\{1,2,\ldots\}$, and let
$R_1,\ldots,R_m$ be averaged nonexpansive operators on $X$.
Then
\begin{equation}
\cran(\Id-R_m\cdots R_1) 
= 
\overline{\ran(\Id-R_1)+\cdots+\ran(\Id-R_m)}.
\end{equation}
\end{proposition}
\begin{proof}
The proof is by induction on $m$. 
The base case, $m=1$, is trivial.
Now assume that the result is true for some integer $m\geq 1$,
and that we are given $m+1$ averaged nonexpansive operators 
$R_1,\ldots,R_{m+1}$ on $X$. 
By Fact~\ref{fact:ave:comp}, $R_m\cdots R_1$ is averaged nonexpansive. 
Applying Theorem~\ref{t:key} to 
$R_m\cdots R_1$ and $R_{m+1}$ we obtain
\begin{align*}
&\hspace{-1cm} \cran(\Id-R_{m+1}\cdots R_1)\\
&= \cran\big(\Id-R_{m+1}(R_m\cdots R_1)\big)\\
&= \overline{\ran(\Id-R_{m+1})+\ran(\Id-R_m\cdots R_1)}
\tag{by Theorem~\ref{t:key}}\\
&= \overline{\ran(\Id-R_{m+1})+\overline{\ran(\Id-R_1)+\cdots + \ran(\Id-R_m)}}
\tag{use inductive hypothesis}\\
&= \overline{\ran(\Id-R_1)+\cdots + \ran(\Id-R_{m+1})},
\end{align*}
and the proof is complete.
\end{proof}

\section{Compositions}
\label{s:compo}

By combining Theorem~\ref{t:key} with Proposition~\ref{p:aug12i},
we are now ready for the main result on compositions. 

\begin{theorem}[{\rm \textbf{main result on compositions}}]
\label{t:aug12ii}
Let $m\in\{1,2,\ldots\}$, and let $R_1,\ldots,R_m$ be nonexpansive on $X$. 
Suppose there exists $j\in\{1,\ldots,m\}$ such that 
each $R_i$ is averaged nonexpansive whenever $i\neq j$. 
Let $\sigma$ be a permutation of $\{1,\ldots,m\}$. 
Then
\begin{subequations}
  \label{e:aug14iandii}
\begin{align}
\cran(\Id-R_m\cdots R_1) 
&= \overline{\ran(\Id-R_1)+\cdots+\ran(\Id-R_m)}\label{e:aug14i}\\
&=\cran\big(\Id-R_{\sigma(m)}\cdots R_{\sigma(1)}\big); \label{e:aug14ii}
\end{align}
\end{subequations}
consequently, 
\begin{subequations}
  \label{e:aug14iiiandiv}
\begin{equation}
  \label{e:aug14iii}
\mdv{R_{\sigma(m)}\cdots R_{\sigma(1)}}
= \mdv{R_m\cdots R_1} 
= P_{\overline{\ran(\Id-R_1)+\cdots+\ran(\Id-R_m)}}(0)
\end{equation}
and
\begin{equation}
\label{e:aug14iv}
\|\mdv{R_{\sigma(m)}\cdots R_{\sigma(1)}}\|
=\|\mdv{R_m\cdots R_1}\| \leq \|\mdv{R_1}\|+\cdots+\|\mdv{R_m}\|. 
\end{equation}
\end{subequations}
\end{theorem}
\begin{proof}
The result is clear when $m\in\{1,2\}$.
So suppose $m\geq 3$.
The conclusion follows in the \emph{boundary case}, i.e., $j=1$ or $j=m$, 
by combining Proposition~\ref{p:aug12i} with Theorem~\ref{t:key}.
So let us assume additionally that $2\leq j \leq m-1$. 
Then $T_m\cdots T_1 = S_2S_1$, where
$S_1 = R_{j-1}\cdots R_1$ is averaged nonexpansive and 
$S_2= R_m\cdots R_j$ is nonexpansive. 
On the one hand, $\cran(\Id-S_1) = \overline{\ran(\Id-R_1)+\cdots +\ran(\Id-R_{j-1})}$
by Proposition~\ref{p:aug12i}.
On the other hand, 
$\cran(\Id-S_2) = \overline{\ran(\Id-R_{j})+\cdots \ran(\Id-R_m)}$
by the boundary case.
Altogether, \eqref{e:aug14i} follows by Theorem~\ref{t:key} 
applied to $S_1$ and $S_2$. 
Finally, \eqref{e:aug14ii} is a direct consequence of \eqref{e:aug14i}
while \eqref{e:aug14iiiandiv} is implied by \eqref{e:aug14iandii}. 
\end{proof}

\begin{remark} Some comments are in order.
  \begin{enumerate}
    \item
    The inequality in \eqref{e:aug14iv}
 massively generalizes
  \cite[Theorem~2.2]{Vic2},
  where each $R_i$ was assumed to be firmly nonexpansive!
  \item 
    The inequality in \eqref{e:aug14iv}
    is sharp
  even when each $R_i$ is firmly nonexpansive; see
  \cite[Example~2.4]{Vic2}. 
  \end{enumerate}
\end{remark}

The following example shows that 
at least one of the mappings
in Theorem~\ref{t:aug12ii} must be averaged. 

\begin{example} 
  \label{ex:all:but:1:fail}
Let $u_1,u_2$ be in $X$, and let $i\in\{1,2\}$.
Set 
$R_i\colon x\mapsto  -x-u_i$, which is nonexpansive but not averaged.
Now $(\forall x\in X)$ we have
$x-R_ix=2x+u_i$. 
Hence
\begin{equation}
\ran(\Id-R_i) = X.
\end{equation}
We also have $(\forall x\in X)$
$R_2R_1x = x+u_1-u_2$
and $R_1R_2x = x+u_2-u_1$. 
It follows that
$(\forall x\in X)$
$x-R_2R_1x = u_2-u_1$
and $x-R_1R_2x = u_1-u_2$.
Thus
$\ran(\Id-R_2R_1) = \{u_2-u_1\}$ and 
$\ran(\Id-R_1R_2) = \{u_1-u_2\}$;
in turn,
\begin{equation}
\mdv{R_2R_1} = u_2-u_1
\;\;\text{yet}\;\;
\mdv{R_1R_2} = u_1-u_2. 
\end{equation}
Therefore, 
the conclusions of Theorem~\ref{t:aug12ii} fail for general nonexpansive mappings
whenever $u_1\neq u_2$. 
\end{example}

When $m=2$, the following positive result for cyclic permutations of the mappings
can be found in \cite[Lemma~2.6]{Sicon2014}:

\begin{proposition}
\label{p:cyclicmagnitude}
Let $m\in\{2,3,\ldots\}$ and let $R_1,\ldots,R_m$ be
nonexpansive on $X$. Then
\begin{equation}
\label{e:aug15ii}
\|\mdv{R_mR_{m-1}\cdots R_1}\| 
= \|\mdv{R_{1}R_m\cdots R_2}\|
= \cdots
= \|\mdv{R_{m-1}\cdots R_1R_m}\|
\end{equation}
\end{proposition}
\begin{proof}
Let $(x_n)_\nnn$ be a sequence such that
$\|x_n-R_mR_{m-1}\cdots R_1x_n\|\to\|\mdv{R_mR_{m-1}\cdots R_1}\|$
Then
$\|\mdv{R_1R_mR_{m-1}\cdots R_2}\|
\leq \|(R_1x_n)-(R_1R_m\cdots R_2)(R_1x_n)\|
=\|R_1x_n-R_1(R_m\cdots R_1x_n)\|
\leq \|x_n-R_m\cdots R_1x_n\|\to\|\mdv{R_mR_{m-1}\cdots R_1}\|$.
Hence
\begin{equation}
  \|\mdv{R_1R_mR_{m-1}\cdots R_2}\|
  \leq 
  \|\mdv{R_mR_{m-1}\cdots R_1}\|
\end{equation}
and the result follows by continuing cyclically in this fashion. 
\end{proof}

\begin{proposition}
  \label{p:backfromvan}
Let $m=3$ and $R_i=\delta_i\Id-a_i$, where
$\delta_i\in\{-1,0,1\}$. 
Then $\ran(\Id-R_i) = \{a_i\}$, if $\delta_i=1$;
and $\ran(\Id-R_i)=X$, if $\delta_i\in\{-1,0\}$. 
Moreover, 
$R_3R_2R_1\colon X\to X\colon x \mapsto 
\delta_3\delta_2\delta_1 x-a_3-\delta_3a_2-\delta_3\delta_2a_1$ and so
\begin{equation}
\ran(\Id-R_3R_2R_1) = 
\begin{cases}
\{a_3+\delta_3a_2+\delta_3\delta_2a_1\}, &\text{if $\delta_1\delta_2\delta_3=1$;}\\
X, &\text{otherwise,}
\end{cases}
\end{equation}
which implies 
\begin{equation}
  \label{e:backfromvan}
\mdv{R_3R_2R_1} = 
\begin{cases}
a_3+\delta_3a_2+\delta_3\delta_2a_1, &\text{if $\delta_1\delta_2\delta_3=1$;}\\
0, &\text{otherwise.}
\end{cases}
\end{equation}
\end{proposition}

\begin{example}
  \label{ex:noncyclic}
Suppose that $m=3$,
$R_1x=-x$, $R_2x=-x+u$,
and $R_3=x-u$, where $u\in X\smallsetminus\{0\}$.
Then $R_1$ and $R_2$ are nonexpansive but not averaged,
while $R_3$ is firmly nonexpansive.
Therefore  
\begin{equation}
\mdv{R_3R_2R_1} = 0
\end{equation}
while 
\begin{equation}
\mdv{R_3R_1R_2} = 2u\neq 0. 
\end{equation}
Hence $\|\mdv{R_3R_2R_1}\|=0<2\|u\|=\|\mdv{R_3R_1R_2}\|$. 
\end{example}
\begin{proof}
Set $\delta_1=-1$, $\delta_2=-1$, $\delta_3=1$, and 
$a_1=0$, $a_2=-u$, $a_3=u$. 
Then $\delta_1\delta_2\delta_3=1$ and 
\eqref{e:backfromvan} yields
\begin{equation}
  \mdv{R_3R_2R_1} = u+1\cdot(-u)+(1)(-1)\cdot 0 = 0. 
\end{equation}
Similarly,
\begin{equation}
  \mdv{R_3R_1R_2} = u+1\cdot 0+(1)(-1)\cdot (-u) = 2u
\end{equation}
and the proof is complete. 
\end{proof}

\begin{remark} \ 
\begin{enumerate}
\item 
Example~\ref{ex:noncyclic} illustrates 
that the assumption that $m-1$ --- and not merely $1$ --- of the operators
in Theorem~\ref{t:aug12ii} be averaged nonexpansive is critical.
\item 
Example~\ref{ex:noncyclic} also shows that 
Proposition~\ref{p:cyclicmagnitude} fails for noncyclic permutations of the mappings.
\end{enumerate}
\end{remark}

\begin{proposition}
  \label{p:aug15iii}
Let $m\in\{2,3,\ldots,\}$, and let $R_1,\ldots,R_m$ be nonexpansive on $X$. 
Then we have the following equivalences:
\begin{subequations}
\begin{align}
\vspace{-1cm}
&\quad\mdv{R_mR_{m-1}\cdots R_1}\in\ran(\Id-R_mR_{m-1}\cdots R_1)\notag\\
&\siff 
\mdv{R_{m-1}\cdots R_1R_m}\in\ran(\Id-R_{m-1}\cdots R_1R_m)\label{e:aug15i}\\
&\siff \cdots\\
&\siff 
\mdv{R_1R_{m}\cdots R_2}\in\ran(\Id-R_1R_{m}\cdots R_2). 
\end{align}
\end{subequations}

\end{proposition}
\begin{proof}
By symmetry, it suffices to show that ``$\RA$'' in \eqref{e:aug15i} holds.
To this end, assume that there is $y\in X$ such that
$\mdv{R_mR_{m-1}\cdots R_1} = y-R_mR_{m-1}\cdots R_1y$.
Either a direct argument or \cite[Proposition~2.5(iv)]{BM2015:AFF} 
shows that 
$\mdv{R_mR_{m-1}\cdots R_1} = (R_mR_{m-1}\cdots R_1)y-(R_mR_{m-1}\cdots R_1)^2y$.
Thus
\begin{align*}
  \|\mdv{R_{m-1}\cdots R_1R_m}\| 
  &\stackrel{\eqref{e:aug15ii}}{=} 
  \|\mdv{R_{m}\cdots R_2R_1}\| \\
  &=\|(R_mR_{m-1}\cdots R_1)y-(R_mR_{m-1}\cdots R_1)^2y\|\\
  &\leq\|R_{m-1}\cdots R_1y-(R_{m-1}\cdots R_1)R_mR_{m-1}\cdots R_1y\|\\
  &\leq \|y-R_mR_{m-1}\cdots R_1y\|\\
  &= \|\mdv{R_mR_{m-1}\cdots R_1}\|\\
  &\stackrel{\eqref{e:aug15ii}}{=} 
  \|\mdv{R_{m-1}\cdots R_1R_m}\| .
\end{align*}
Consequently, $\|\mdv{R_{m-1}\cdots R_1R_m}\| 
= \|R_{m-1}\cdots R_1y-(R_{m-1}\cdots R_1R_m)R_{m-1}\cdots R_1y\|$.
Therefore 
\begin{align*}
  \mdv{R_{m-1}\cdots R_1R_m}
  &= R_{m-1}\cdots R_1y-(R_{m-1}\cdots R_1R_m)R_{m-1}\cdots R_1y\\
  &\in\ran(\Id-R_{m-1}\cdots R_1R_m),
\end{align*}
and the proof is complete. 
\end{proof}

\begin{remark}
  \label{rem:De:Pierro}
  Proposition~\ref{p:aug15iii}
  shows that the minimal displacement vector is attained
  for all cyclic shifts of the composition.
  For noncyclic shifts, this result goes wrong
  as De Pierro observed in \cite[Section~3 on page~193]{DeP2000}
  (see also \cite[Example~2.7]{BM2015:AFF}).
\end{remark}

Let us conclude this section with an application to
the projected gradient descent method (see also \cite{Moursi17} for 
an analysis of the forward-backward method in the possibly inconsistent case). 

\begin{corollary}[{\rm \bf projected gradient descent}]
Let $f\colon X\to\RR$ be convex and differentiable on $X$,
with $\nabla f$ being $L$-Lipschitz continuous, let
$C$ be a nonempty closed convex subset of $X$, 
let $\alpha\in\left]0,2\right[$. 
Then the magnitude of the minimal dispacement vector
of the projected gradient descent operator
\begin{equation}
  T\colon X\to X \colon x\mapsto \Pj{C}\circ\big(\Id-\alpha\tfrac{1}{L}\nabla f\big)
\end{equation}
satisfies
$\|\mdv{T}\|\leq \alpha{L}^{-1}\inf\|\nabla f(X)\|$. 
\end{corollary}
\begin{proof}
This follows from Theorem~\ref{t:aug12ii} with $m=2$,
$R_1 = \Id-\alpha L^{-1}\nabla f$ and $R_2=P_C$, where $\mdv{R_2}=0$.
Hence $\|\mdv{R_2}\|=0$ and 
$\|\mdv{R_1}\|=\inf_{x\in X}\|x-R_1x\|=\inf_{x\in X}\|\alpha L^{-1}\nabla f(x)\|$. 
Now use \eqref{e:aug14iv}. 
\end{proof}

\section{Convex combinations}
\label{s:convco}

In this final section, we focus on convex combinations of nonexpansive mappings.

\begin{theorem}[{\rm \textbf{main result on convex combinations}}]
  \label{thm:main:convc}
  Let $m\in\{2,3,\ldots\}$, let
  $R_1,\ldots,R_m$ be nonexpansive on $X$,
  and let $\lambda_1,\ldots,\lambda_m$ be in
  $\left]0,1\right[$ such that $\sum_{i=1}^m  \lambda_i =1$.
  Set 
  \begin{equation}
    \overline{R} = \sum_{i=1}^m\lambda_iR_i.
  \end{equation}
  Then
  \begin{equation}
    \label{e:aug16i}
    \cran(\Id-\overline{R}) = 
    \overline{\textstyle\sum_{i=1}^m\lambda_i\ran(\Id-R_i)}.
  \end{equation}
  Consequently, 
  \begin{equation}
    \label{e:aug16ii}
    \mdv{\overline{R}}
    = \Pj{\overline{\sum_{i=1}^m\lambda_i\ran(\Id-R_i)}}(0)
  \end{equation}
  and
  \begin{equation}
    \label{e:aug16iii}
    \|\mdv{\overline{R}}\| \leq 
    \|\textstyle\sum_{i=1}^m\lambda_i \mdv{R_i}\|
    \leq \textstyle\sum_{i=1}^m\lambda_i \|\mdv{R_i}\|.
  \end{equation}
\end{theorem}
\begin{proof}
Set $A_i := \Id-R_i$ for each $i\in\{1,\ldots,m\}$.
By \cite[Example~20.29 and Example~25.20]{BC2017},
$A_i$ is maximally and \reckti.
By \cite[Lemma~3.1]{Vic2},
\begin{align*}
  \cran(\Id-\overline{R})
  &= \cran\big(\textstyle\sum_{i=1}^m\lambda_i(\Id-R_i)\big)
  = \cran\big(\textstyle\sum_{i=1}^m\lambda_iA_i\big)
  =\overline{\sum_{i=1}^m\lambda_i\ran A_i}\\
  &=\overline{\textstyle\sum_{i=1}^m\lambda_i\ran(\Id-R_i)}.
\end{align*}
This yields \eqref{e:aug16i} and thus \eqref{e:aug16ii}.
In view of \eqref{e:aug16i}, we have
\begin{equation}
\textstyle \sum_{i=1}^m \lambda_i\mdv{R_i}
\in\sum_{i=1}^m \lambda_i\cran(\Id-R_i)
\subseteq \overline{\sum_{i=1}^m \lambda_i\ran(\Id-R_i)}
=\cran(\Id-\overline{R}).
\end{equation}
Thus $\|\mdv{\overline{R}}\|\leq\|\sum_{i=1}^m\|\lambda_i\mdv{R_i}\|$
and \eqref{e:aug16iii} follows. 
\end{proof}

\begin{remark}
  \cite[Example~3.4 and Example~3.5]{Vic2} illustrate that
  the inequalities in \eqref{e:aug16iii} are sharp and also
  that in general $\mdv{\overline{R}} \neq \sum_{i=1}^m\lambda_i\mdv{R_i}$.
\end{remark}

We conclude with an application that extends 
\cite[Theorem~5.5]{Vic1} and \cite[Corollary~3.3]{Vic2}, 
where each $\mdv{R_i}$ was equal to $0$:

\begin{corollary}
Let $m\in\{2,3,\ldots\}$, let $R_1,\ldots,R_m$ be
nonexpansive on $X$, and let $\lambda_1,\ldots,\lambda_m$
be in $\left]0,1\right[$ such that
$\sum_{i=1}^{m}\lambda_i=1$. 
Set $\overline{R}=\sum_{i=1}^m \lambda_iR_i$ 
and assume that $\sum_{i=1}^m\lambda_i\mdv{R_i}=0$.
Then $\mdv{\bar{R}}=0$.
\end{corollary}
\begin{proof}
Clear from \eqref{e:aug16iii}. 
\end{proof}

\section*{Acknowledgments}
The research of HHB was partially supported by a Discovery Grant
of the Natural Sciences and Engineering Research Council of
Canada. 
The research of WMM was partially supported by 
the Natural Sciences and Engineering Research Council of
Canada Postdoctoral Fellowship.

\end{document}